\documentclass[11pt,reqno]{article}
\usepackage{amsmath}
\usepackage{amssymb}
\usepackage{amsthm}
\usepackage[usenames]{color}
\usepackage{amscd}
\usepackage{dsfont}
\usepackage{indentfirst}

\usepackage[colorlinks=true,linkcolor=blue,filecolor=red,
citecolor=webgreen]{hyperref}
\definecolor{webgreen}{rgb}{0,.5,0}

\usepackage{chngcntr}
\counterwithout{equation}{section}

\def\C{{\mathds{C}}}

\def\N{{\mathds{N}}}
\def\Z{{\mathds{Z}}}

\def\1{{\bf 1}}

\def\lcm{\operatorname{lcm}}

\newtheorem{theorem}{Theorem}

\newtheorem{lemma}[theorem]{Lemma}
\newtheorem{corollary}[theorem]{Corollary}

\newtheorem{remark}[theorem]{Remark}

\begin{document}

\title{{\bf Short proofs, generalizations and applications of certain identities concerning 
multiple Dirichlet series}}
\author{L\'aszl\'o T\'oth  \\
Department of Mathematics \\
University of P\'ecs \\
Ifj\'us\'ag \'utja 6, 7624 P\'ecs, Hungary \\
E-mail: {\tt ltoth@gamma.ttk.pte.hu}}
\date{}
\maketitle

\centerline{Journal of Integer Sequences, Vol. 26 (2023), Article 23.2.1}

\begin{abstract} We present other proofs, generalizations and analogues of the identities 
concerning multiple Dirichlet series by Tahmi and Derbal (2022). As applications, we
obtain asymptotic formulas with remainder terms for certain related sums. 
\end{abstract}



\section{Introduction}
Recently, Tahmi and Derbal \cite{TahDer2022} obtained certain identities for the
multiple Dirichlet series
\begin{equation*}
\sum_{\substack{n_1,\ldots,n_r=1\\ \gcd(n_1,\ldots,n_r)=1}}^{\infty} 
\frac{f(n_1)\cdots f(n_r)}{n_1^{s_1}\cdots n_r^{s_r}},
\end{equation*}
with $r\ge 2$ an integer, in the cases where $f:\N:=\{1,2,\ldots\} \to \C$ is a completely multiplicative arithmetic function or the
Dirichlet convolution of two completely multiplicative functions.

As direct corollaries of their results, they mentioned, among others, the identities  
\begin{equation*} 
\sum_{\substack{n_1,\ldots,n_r=1\\ \gcd(n_1,\ldots,n_r)=1}}^{\infty} 
\frac{1}{n_1^{s_1}\cdots n_r^{s_r}} = \frac{\zeta(s_1)\cdots \zeta(s_r)}{\zeta(s_1+\cdots+s_r)},
\end{equation*}
with $s_i\in \C$, $\Re s_i>1$ ($1\le i\le r$), concerning the Dirichlet series of the characteristic 
function of the set of points in $\N^r$, which are visible from the origin (also see Apostol \cite[p.\ 248]{Apo1976}), 
and 
\begin{equation} \label{tau_2}
\sum_{\substack{n_1,n_2=1\\ \gcd(n_1,n_2)=1}}^{\infty} 
\frac{\tau(n_1)\tau(n_2)}{n_1^{s_1}n_2^{s_2}} = \zeta^2(s_1)\zeta^2(s_2)\prod_p \Bigl( 1-\frac{4}{p^{s_1+s_2}} +\frac{2}{p^{2s_1+s_2}}+\frac{2}{p^{s_1+2s_2}}-\frac1{p^{2s_1+2s_2}} \Bigr),
\end{equation}
with $s_i\in \C$, $\Re s_i>1$ ($1\le i\le 2$), where $\tau(n)=\sum_{d\mid n} 1$ is the divisor function.

The proofs given in \cite{TahDer2022} are by using Euler product expansions of the Dirichlet series of 
some appropriate multiplicative functions of one variable.

In this paper we present other proofs and generalizations of the results by Tahmi and Derbal \cite{TahDer2022} by considering 
Euler product expansions of some multiple Dirichlet series
\begin{equation*}
\sum_{\substack{n_1,\ldots,n_r=1}}^{\infty} \frac{F(n_1,\ldots,n_r)}{n_1^{s_1}\cdots n_r^{s_r}}
\end{equation*}
of multiplicative arithmetic functions $F:\N^r\to \C$ of $r$ variables. Namely, we investigate the functions
\begin{equation} \label{F_rel_prime}
F(n_1,\ldots,n_r)= 
\begin{cases}
f_1(n_1)\cdots f_r(n_r), & \text{if $\gcd(n_1,\ldots,n_r)=1$;}\\
0, & \text{otherwise,}
\end{cases}
\end{equation}
where each of the functions $f_1,\ldots, f_r:\N \to \C$ is the Dirichlet convolution of $t$ ($t\ge 1$) completely multiplicative functions. 
Note that if $f_1,\ldots,f_r$ are multiplicative, then $F$ given by \eqref{F_rel_prime} is multiplicative, viewed as a function of $r$ variables.

We also make more explicit the formula of \cite[Th.\ 3.2]{TahDer2022}, as applications we
obtain asymptotic formulas with remainder terms for certain sums 
\begin{equation*}
\sum_{\substack{n_1,\ldots,n_r\le x \\ \gcd(n_1,\ldots,n_r)=1}} f_1(n_1)\cdots f_r(n_r),
\end{equation*}
and also derive similar results where the condition $\gcd(n_1,\ldots,n_r)=1$ is replaced by the condition that $n_1,\ldots,n_r$ are pairwise relatively prime.
Some basic properties of multiplicative arithmetic functions of $r$ variables are reviewed in Section \ref{Section_Prelim_1}. Certain polynomial identities needed in the proofs are given in Section \ref{Section_Prelim_2}.
Our main results and their proofs on the Dirichlet series are included in Section \ref{Section_Dirichlet_series}, and some
related asymptotic formulas are presented in Section \ref{Section_Asymptotic_formulas}.
All the identities regarding Dirichlet series and Euler
products are considered formally or in the case of absolute convergence.

\section{Preliminaries} 

\subsection{Arithmetic functions of several variables} \label{Section_Prelim_1}

Let $F:\N^r\to \C$ be an arbitrary arithmetic function of $r$ variables ($r\ge 1$). Its Dirichlet series is given by
\begin{equation*}
D(F,s_1,\ldots,s_r): = \sum_{n_1,\ldots,n_r=1}^{\infty}
\frac{F(n_1,\ldots,n_r)}{n_1^{s_1}\cdots n_r^{s_r}}.
\end{equation*}

The Dirichlet convolution of the functions $F,G:\N^r\to \C$ is defined by
\begin{equation*}
(F*G)(n_1,\ldots,n_r)= \sum_{d_1\mid n_1, \ldots, d_r\mid n_r}
F(d_1,\ldots,d_r) G(n_1/d_1, \ldots, n_r/d_r).
\end{equation*}

If $D(F,s_1,\ldots,s_r)$ and $D(G,s_1,\ldots,s_r)$, with $s_1,\ldots, s_r\in \C$, are
absolutely convergent, then
$D(F*G;s_1,\ldots,s_r)$ is also absolutely convergent and
\begin{equation*}
D(F*G, s_1,\ldots,s_r) = D(F,s_1,\ldots,s_r) D(G,s_1,\ldots,s_r).
\end{equation*}

A nonzero arithmetic function $F:\N^r\to \C$ is said to be {\it multiplicative\/} if
\begin{equation*}
F(m_1n_1,\ldots,m_rn_r)= F(m_1,\ldots,m_r) F(n_1,\ldots,n_r)
\end{equation*}
holds for every $m_1,\ldots,m_r,n_1,\ldots,n_r\in \N$ such that $\gcd(m_1\cdots m_r,n_1\cdots n_r)=1$.
If $F$ is multiplicative, then it is determined by the values
$F(p^{a_1},\ldots,p^{a_r})$, where $p$ is prime and
$a_1,\ldots,a_r\in \N_0:=\N \cup \{0\}$. More exactly, $F(1,\ldots,1)=1$ and
for every $n_1,\ldots,n_r\in \N$,
\begin{equation*}
F(n_1,\ldots,n_r)= \prod_p F\bigl(p^{a_p(n_1)}, \ldots,p^{a_p(n_r)}\bigr),
\end{equation*}
by using the notation $n=\prod_p p^{a_p(n)}$ for the prime power factorization of $n\in \N$, the product being over the primes $p$,
where all but a finite number of the exponents $a_p(n)$ are zero. 

Examples of multiplicative functions of $r$ variables are the GCD and LCM functions $\gcd(n_1,\ldots,n_r)$, $\lcm(n_1,\ldots,n_r)$
and the characteristic functions
\begin{align*}
\varrho(n_1,\ldots,n_r) &= 
\begin{cases}
1, & \text{if $\gcd(n_1,\ldots,n_r)=1$;}\\
0, & \text{otherwise,} 
\end{cases}\\
\vartheta(n_1,\ldots,n_r) &= 
\begin{cases}
1, & \text{if $\gcd(n_i,n_j)=1$ for every $1\le i< j\le r$;}\\
0, & \text{otherwise.}
\end{cases}
\end{align*}

If $F,G:\N^r \to \C$ are multiplicative, then their Dirichlet convolution $F*G$ is also multiplicative.
If $F$ is multiplicative, then its Dirichlet series can be expanded into a (formal)
Euler product, that is,
\begin{equation} \label{Euler_product} 
D(F,s_1,\ldots,s_r)=  \prod_p \sum_{a_1,\ldots,a_r=0}^{\infty}
\frac{f(p^{a_1},\ldots, p^{a_r})}{p^{a_1s_1+\cdots +a_r s_r}},
\end{equation}
the product being over the primes $p$. More exactly, if $F$ is multiplicative, then the series
$D(F,s_1,\ldots,s_r)$ with $s_1,\ldots,s_r\in\C$ is absolutely convergent if and only if
\begin{equation*}
\sum_p \sum_{\substack{a_1,\ldots,a_r=0\\ a_1+\cdots +a_r \ge 1}}^{\infty}
\frac{|f(p^{a_1},\ldots, p^{a_r})|}{p^{a_1 \Re s_1+\cdots +a_r \Re s_r}} < \infty
\end{equation*}
and in this case equality \eqref{Euler_product} holds. 

See, e.g., Delange \cite{Del1969} and the survey by the author \cite{Tot2014} for these and 
some related results on arithmetic functions of $r$ variables. If $r=1$, i.e., in the case of functions 
of a single variable we recover some familiar properties.

\subsection{Some polynomial identities} \label{Section_Prelim_2}

Let $e_j(x_1,\ldots,x_t)=\sum_{1\le i_1<\cdots <i_j\le t} x_{i_1}\cdots x_{i_j}$
denote the elementary symmetric polynomials in $x_1,\ldots, x_t$ of degree $j$ ($1\le j \le t$).
We will use the polynomial identity 
\begin{equation} \label{pol_id}
P(x): = \prod_{j=1}^t (x-x_j)= x^t+ \sum_{j=1}^t (-1)^j e_j(x_1,\ldots, x_t) x^{t-j}.
\end{equation}

Taking derivatives gives 
\begin{equation} \label{pol_id_der}
P'(x)= \sum_{j=1}^t \prod_{\substack{k=1\\ k\ne j}}^t (x-x_k) = t x^{t-1}+ \sum_{j=1}^{t-1} (-1)^j(t-j) e_j(x_1,\ldots, x_t) x^{t-j-1}.
\end{equation}

We need the following lemma.

\begin{lemma} \label{Lemma} If $t\in \N$ and $x_1,\ldots,x_t\in \C$, then
\begin{equation*}
(1-t)\prod_{j=1}^t (1-x_j) + \sum_{j=1}^t \prod_{\substack{k=1\\ k\ne j}}^t (1-x_k) = 1+ \sum_{j=2}^t (-1)^{j-1}(j-1)e_j(x_1,\ldots,x_t).    
\end{equation*}
\end{lemma}

\begin{proof} By using \eqref{pol_id} and \eqref{pol_id_der},
\begin{equation*}
(1-t)\prod_{j=1}^t (1-x_j) + \sum_{j=1}^t \prod_{\substack{k=1\\ k\ne j}}^t (1-x_k) = (1-t)P(1)+P'(1)     
\end{equation*}
\begin{equation*}
= (1-t)\Bigl(1+ \sum_{j=1}^t (-1)^j e_j(x_1,\ldots, x_t)\Bigr)+ t + \sum_{j=1}^{t-1} (-1)^j (t-j) e_j(x_1,\ldots, x_t) 
\end{equation*}
\begin{equation*}
= 1+ \sum_{j=2}^t (-1)^{j-1}(j-1)e_j(x_1,\ldots,x_t).    
\end{equation*}
\end{proof}

\section{Identities for Dirichlet series} \label{Section_Dirichlet_series}

Our first result is the following. As above, let $*$ denote the Dirichlet convolution of arithmetic functions and let $D(f,s):=\sum_{n=1}^{\infty}
f(n)n^{-s}$ stand for the Dirichlet series of the function $f:\N\to \C$. We recall that a nonzero function $f:\N \to \C$ is {\it completely multiplicative\/} if
$f(mn)=f(m)f(n)$ holds for all $m,n\in \N$.

\begin{theorem} \label{Th_1}
Let $r\ge 2$, $t\ge 1$ be fixed integers and let $f_i= g_{i1}*\cdots *g_{it}$, where $g_{ij}:\N \to \C$ are nonzero completely 
multiplicative functions \textup{($1\le i\le r$, $1\le j\le t$)}.
Then we have (formally or in the case of absolute convergence),
\begin{equation*}
\sum_{\substack{n_1,\ldots,n_r=1\\ \gcd(n_1,\ldots,n_r)=1}}^{\infty} 
\frac{f_1(n_1)\cdots f_r(n_r)}{n_1^{s_1}\cdots n_r^{s_r}} = D(f_1,s_1)\cdots D(f_r,s_r) \Delta(f_1,\ldots,f_r,s_1,\ldots,s_r),
\end{equation*}
where
\begin{align}
& \Delta(f_1,\ldots,f_r,s_1,\ldots,s_r) =  \nonumber \\
& \quad\quad \prod_p \Bigl(1+ (-1)^{r-1} \sum_{1\le a_1,\ldots,a_r\le t} \frac1{p^{a_1s_1+\cdots +a_rs_r}} \prod_{i=1}^r (-1)^{a_i} G_{i a_i}(p)\Bigr),
\label{id_Th_1}
\end{align}
and 
\begin{equation} \label{G_i_j_p}
G_{ij}(p):= \sum_{1\le \ell_1<\cdots <\ell_j\le t} g_{i\ell_1}(p)\cdots g_{i\ell_j}(p),
\end{equation}
with $1\le i\le r$, $1\le j\le t$, and $p$ a prime.
\end{theorem}

In the cases $t=1$ and $t=2$, with $f_1=\cdots =f_r=f$, Theorem \ref{Th_1} recovers \cite[Ths.\ 3.1, 3.2]{TahDer2022}. 

\begin{proof}
As mentioned above, the characteristic function $\varrho$ of the $r$-tuples with relatively prime components
is multiplicative, viewed as a functions of $r$ variables. Note that for primes $p$ and $a_1,\ldots,a_r \ge 0$ 
we have $\varrho(p^{a_1},\ldots,p^{a_r})=1$ if and only if there is at least one $a_i=0$. Also, 
if $f_1,\ldots,f_r:\N\to \C$ are arbitrary 
multiplicative functions of a single variable, then their product $f_1(n_1)\cdots f_r(n_r)$ is multiplicative 
as a function of $r$ variables. We deduce the Euler product expansion
\begin{equation*}
D:= \sum_{\substack{n_1,\ldots,n_r=1\\ \gcd(n_1,\ldots,n_r)=1}}^{\infty} 
\frac{f_1(n_1)\cdots f_r(n_r)}{n_1^{s_1}\cdots n_r^{s_r}} = \sum_{n_1,\ldots,n_r=1}^{\infty} 
\frac{f_1(n_1)\cdots f_r(n_r)\varrho(n_1,\ldots,n_r)}{n_1^{s_1}\cdots n_r^{s_r}}
\end{equation*}
\begin{equation*}
= \prod_p \sum_{a_1,\ldots,a_r=0}^{\infty} \frac{f_1(p^{a_1})\cdots f_r(p^{a_r})\varrho(p^{a_1},\ldots,p^{a_r})}{p^{a_1s_1+\cdots + a_rs_r}}    
\end{equation*}
\begin{equation*}
= \prod_p \sum_{\substack{a_1,\ldots,a_r=0\\ a_1\cdots a_r=0}}^{\infty} \frac{f_1(p^{a_1})\cdots f_r(p^{a_r})}{p^{a_1s_1+\cdots + a_rs_r}}    
\end{equation*}
\begin{equation*}
= \prod_p \biggl( \sum_{a_1,\ldots,a_r=0}^{\infty} - \sum_{a_1,\ldots,a_r=1}^{\infty}\biggr) 
\frac{f_1(p^{a_1})\cdots f_r(p^{a_r})}{p^{a_1s_1+\cdots + a_rs_r}}.    
\end{equation*}

Now if $f_i= g_{i1}*\cdots *g_{it}$, where $g_{ij}:\N \to \C$ are completely 
multiplicative functions ($1\le i\le r$, $1\le j\le t$), then  
\begin{equation*}
\sum_{n=1}^{\infty} \frac{f_i(n)}{n^s} = \prod_{j=1}^t \sum_{n=1}^{\infty} \frac{g_{ij}(n)}{n^s} 
=\prod_{j=1}^t \prod_p \Bigl(1-\frac{g_{ij}(p)}{p^s}\Bigr)^{-1} =\prod_p \prod_{j=1}^t  \Bigl(1-\frac{g_{ij}(p)}{p^s}\Bigr)^{-1}.
\end{equation*}

At the same time, since the functions $f_i$ ($1\le i\le r$) are multiplicative, we have
\begin{equation*}
\sum_{n=1}^{\infty} \frac{f_i(n)}{n^s} =  \prod_p \sum_{a=0}^{\infty} \frac{f_i(p^a)}{p^{as}},
\end{equation*}
showing that
\begin{equation*}
\sum_{a=0}^{\infty} \frac{f_i(p^a)}{p^{as}}= \prod_{j=1}^t  \Bigl(1-\frac{g_{ij}(p)}{p^s}\Bigr)^{-1},
\end{equation*}
and
\begin{equation} \label{sum_f_1}
\sum_{a=1}^{\infty} \frac{f_i(p^a)}{p^{as}}= \prod_{j=1}^t  \Bigl(1-\frac{g_{ij}(p)}{p^s}\Bigr)^{-1} -1.
\end{equation}

It follows that
\begin{align} \nonumber
D &= \prod_p \biggl(
\prod_{i=1}^r \prod_{j=1}^t  \Bigl(1-\frac{g_{ij}(p)}{p^{s_i}}\Bigr)^{-1} - 
\prod_{i=1}^r \biggl(\prod_{j=1}^t \Bigl(1-\frac{g_{ij}(p)}{p^{s_i}} \Bigr)^{-1} -1 \biggr)   
\biggr) \\      
 \label{D_form_one}
& = \prod_p \prod_{i=1}^r \prod_{j=1}^t  \Bigl(1-\frac{g_{ij}(p)}{p^{s_i}}\Bigr)^{-1} 
\prod_p \biggl(1- \prod_{i=1}^r \biggl(1- \prod_{j=1}^t \Bigl(1-\frac{g_{ij}(p)}{p^{s_i}} \Bigr) \biggr)   
\biggr).     
\end{align}

Using identity \eqref{pol_id} for $x=1$ and $x_j= g_{ij}(p) p^{-s_i}$ ($1\le j\le t$) we have
\begin{equation*}
1- \prod_{j=1}^t \Bigl(1-\frac{g_{ij}(p)}{p^{s_i}}\Bigr) = \sum_{j=1}^t \frac{(-1)^{j-1}}{p^{js_i}} \sum_{1\le \ell_1< \cdots< \ell_j\le t}
g_{i\ell_1}(p)\cdots g_{i\ell_j}(p),
\end{equation*}
and inserting into \eqref{D_form_one} we deduce
\begin{equation*}
D= D(f_1,s_1)\cdots D(f_r,s_r) 
\prod_p \biggl(1-\prod_{i=1}^r \sum_{j=1}^t \frac{(-1)^{j-1}}{p^{js_i}} G_{ij}(p) \biggr),
\end{equation*}
where $G_{ij}(p)$ is defined by \eqref{G_i_j_p}. Here the product over the primes $p$ is
\begin{equation*}
\prod_p \biggl(1- \biggl(\sum_{a_1=1}^t \frac{(-1)^{a_1-1}}{p^{a_1s_1}} G_{1a_1}(p)\biggr) \cdots 
\biggl( \sum_{a_r=1}^t \frac{(-1)^{a_r-1}}{p^{a_rs_r}} G_{ra_r}(p)\biggr)  \biggr)
\end{equation*}
\begin{equation*}
= \prod_p \biggl(1+ (-1)^{r-1} \sum_{1\le a_1,\ldots,a_r\le t} \frac1{p^{a_1s_1+\cdots +a_rs_r}} \prod_{i=1}^r (-1)^{a_i} G_{i a_i}(p)\biggr),
\end{equation*}
finishing the proof.
\end{proof}

\begin{remark} {\rm Identity \eqref{id_Th_1} shows that under the assumptions of Theorem \ref{Th_1} we have the convolutional identity
\begin{equation} \label{convo_id_1}
f_1(n_1)\cdots f_r(n_r)\varrho(n_1,\ldots,n_r) = \sum_{d_1e_1=n_1,\ldots, d_re_r=n_r} f_1(d_1)\cdots f_r(d_r)  
F_{f_1,\ldots,f_r}(e_1,\ldots,e_r),
\end{equation}
where $F_{f_1,\ldots,f_r}$ is the multiplicative function 
defined for prime powers $p^{a_1}, \ldots, p^{a_r}$ ($a_1, \ldots$, $a_r\ge 0$, 
not all zero) by
\begin{equation*}
F_{f_1,\ldots,f_r}(p^{a_1},\ldots,p^{a_r}) = \begin{cases} (-1)^{r-1} \prod_{i=1}^r (-1)^{a_i} G_{i a_i}(p), & \text{if $1\le a_1,\ldots, a_r\le t$;} 
\\ 0, & \text{ otherwise.} \end{cases}
\end{equation*}
}
\end{remark}

Let $\tau_t(n)= \sum_{d_1\cdots d_t=n} 1$ denote the Piltz divisor function of order $t$.

\begin{corollary} \label{Cor_tau_gen}
Let $r\ge 2$ and let $t_i\ge 2$ \textup{($1\le i\le r$)} be fixed integers. If $s_i\in \C$, $\Re s_i>1$ 
\textup{($1\le i\le r$)}, then 
\begin{equation*}
\sum_{\substack{n_1,\ldots,n_r=1\\ \gcd(n_1,\ldots,s_r)=1}}^{\infty} 
\frac{\tau_{t_1}(n_1)\cdots \tau_{t_r}(n_r)}{n_1^{s_1}\cdots n_r^{s_r}} = \zeta^{t_1}(s_1)\cdots \zeta^{t_r}(s_r) \Delta(\tau_{t_1},\ldots,\tau_{t_r},s_1,\ldots,s_r),
\end{equation*}
with
\begin{equation} \label{Delta_tau}
\Delta(\tau_{t_1},\ldots,\tau_{t_r},s_1,\ldots,s_r)= \prod_p \biggl(1+(-1)^{r-1} \sum_{1\le a_1,\ldots,a_r\le t} \frac1{p^{a_1s_1+\cdots +a_rs_r}} \prod_{i=1}^r (-1)^{a_i} 
\binom{t_i}{a_i}\biggr),
\end{equation}
where $\binom{t_i}{a_i}$ are binomial coefficients with the usual convention that $\binom{t_i}{a_i}=0$
for $a_i>t_i$ \textup{($1\le i\le r$)}.
\end{corollary}

\begin{proof} Choose $g_{ij}(p)=1$ for $1\le j\le t_i$ and $g_{ij}(p)=0$ for $t_i+1\le j\le t$ ($1\le i\le r$). Then $f_i(n)=\tau_{t_i}(n)$ ($1\le i\le r$),
and use that $G_{ij}(p)=\binom{t_i}{j}$  ($1\le i\le r$, $1\le j\le t$) for a prime $p$.
\end{proof}

\begin{corollary} Let $r\ge 2$. If $s_i\in \C$, $\Re s_i>1$ \textup{($1\le i\le r$)}, then 
\begin{equation*}
\sum_{\substack{n_1,\ldots,n_r=1\\ \gcd(n_1,\ldots,n_r)=1}}^{\infty} 
\frac{\tau(n_1)\cdots \tau(n_r)}{n_1^{s_1}\cdots n_r^{s_r}} = \zeta^2(s_1)\cdots \zeta^2(s_r) 
\end{equation*}
\begin{equation*}
\times \prod_p \biggl(1+(-1)^{r-1} \sum_{1\le a_1,\ldots,a_r\le 2} \frac{(-2)^{\# \{1\le i\le r:\, a_i=1\}}}{p^{a_1s_1+\cdots +a_rs_r}}\biggr).
\end{equation*}
\end{corollary}

\begin{proof} Apply Corollary \ref{Cor_tau_gen} for $t_i=2$ ($1\le i\le r$).
\end{proof}

If $r=2$, then this recovers identity \eqref{tau_2} and for $r=3$ we have
\begin{align}
& \sum_{\substack{n_1,n_2,n_3=1\\ \gcd(n_1,n_2,n_3)=1}}^{\infty} 
\frac{\tau(n_1)\tau(n_2)\tau(n_3)}{n_1^{s_1}n_2^{s_2}n_3^{s_3}} = \zeta^2(s_1)\zeta^2(s_2)\zeta^2(s_3)  \nonumber\\
& \quad\quad
\times \prod_p \Bigl(1-\frac{8}{p^{s_1+s_2+s_3}}+ \frac{4}{p^{2s_1+s_2+s_3}} +\frac{4}{p^{s_1+2s_2+s_3}} +\frac{4}{p^{s_1+s_2+2s_3}} \Bigr.  \nonumber \\
& \quad\quad \Bigl. - \frac{2}{p^{2s_1+2s_2+s_3}} - \frac{2}{p^{2s_1+s_2+2s_3}} -\frac{2}{p^{s_1+2s_2+2s_3}} + \frac1{p^{2s_1+2s_2+2s_3}} \Bigr).
\label{prod_tau_prime_3}
\end{align}

Now we consider Dirichlet series with $\gcd(n_1,\ldots,n_r)=1$ replaced by the condition that $n_1,\ldots,n_r$ are pairwise relatively prime.

\begin{theorem} \label{Th_2}
Let $r\ge 2$, $t\ge 1$ be fixed integers and let $f_i= g_{i1}*\cdots *g_{it}$, where $g_{ij}:\N \to \C$ are nonzero completely 
multiplicative functions \textup{($1\le i\le r$, $1\le j\le t$)}. Then we have (formally or in the case of absolute convergence),
\begin{equation*}
\sum_{\substack{n_1,\ldots,n_r=1\\ \gcd(n_i,n_j)=1 \, (i\ne j)}}^{\infty} 
\frac{f_1(n_1)\cdots f_r(n_r)}{n_1^{s_1}\cdots n_r^{s_r}} = D(f_1,s_1)\cdots D(f_r,s_r) \overline{\Delta}(f_1,\ldots,f_r,s_1,\ldots,s_r),
\end{equation*}
where
\begin{align}
& \overline{\Delta}(f_1,\ldots,f_r,s_1,\ldots,s_r)
= \nonumber\\
& \quad\quad \prod_p \biggl(1- \sum_{i=2}^r (i-1) \sum_{1\le \ell_1<\cdots <\ell_i\le r} \sum_{a_{\ell_1},\ldots,a_{\ell_i}=1}^t 
\frac1{p^{a_{\ell_1}s_{\ell_1}+\cdots +a_{\ell_i}s_{\ell_i}}} \prod_{m=1}^i (-1)^{a_{\ell_m}} G_{i a_{\ell_m}}(p)\biggr),
\label{id_Th_2}
\end{align}
with $G_{ij}(p)$ \textup{($1\le i\le r$, $1\le j\le t$, $p$ prime)} 
defined by \eqref{G_i_j_p}.  
\end{theorem}

\begin{proof} The characteristic function $\vartheta$ of the $r$-tuples with pairwise relatively prime components
is multiplicative, viewed as a functions of $r$ variables. Note that for primes $p$ and $a_1,\ldots,a_r \ge 0$ 
we have $\vartheta(p^{a_1},\ldots,p^{a_r})=1$ if and only if there is at most one $a_i\ge 1$. If 
$f_1,\ldots,f_r:\N\to \C$ are arbitrary 
multiplicative functions of a single variable, then we have the Euler product expansion
\begin{equation*}
\overline{D}:= \sum_{\substack{n_1,\ldots,n_r=1\\ \gcd(n_i,n_j)=1\, (i\ne j)}}^{\infty} 
\frac{f_1(n_1)\cdots f_r(n_r)}{n_1^{s_1}\cdots n_r^{s_r}} = \sum_{n_1,\ldots,n_r=1}^{\infty} 
\frac{f_1(n_1)\cdots f_r(n_r)\vartheta(n_1,\ldots,n_r)}{n_1^{s_1}\cdots n_r^{s_r}}
\end{equation*}
\begin{equation*}
= \prod_p \sum_{a_1,\ldots,a_r=0}^{\infty} \frac{f_1(p^{a_1})\cdots f_r(p^{a_r})\vartheta(p^{a_1},\ldots,p^{a_r})}{p^{a_1s_1+\cdots + a_rs_r}}    
= \prod_p \biggl(1+ \sum_{i=1}^r \sum_{a_i=1}^{\infty} \frac{f_i(p^{a_i})}{p^{a_is_i}}\biggr)
\end{equation*}
\begin{equation*}
= \prod_p \biggl(1+ \sum_{i=1}^r \biggl(\prod_{j=1}^t \Bigl(1-\frac{g_{ij}(p)}{p^{s_i}}\Bigr)^{-1}-1\biggr)\biggr)
\end{equation*}

by using \eqref{sum_f_1}. We deduce that
\begin{equation*} 
\overline{D} = \prod_p \prod_{i=1}^r \prod_{j=1}^t  \Bigl(1-\frac{g_{ij}(p)}{p^{s_i}}\Bigr)^{-1} 
\prod_p \biggl((1-r) \prod_{i=1}^r \prod_{j=1}^t \Bigl(1-\frac{g_{ij}(p)}{p^{s_i}} \Bigr) + \sum_{i=1}^r \prod_{\substack{k=1\\ k\ne i}}^r  
\prod_{j=1}^t \Bigl(1-\frac{g_{kj}(p)}{p^{s_k}} \Bigr) \biggr)    
\end{equation*}
\begin{equation} \label{express} 
= \prod_{i=1}^r D(f_i,s_i) \prod_p K(p),     
\end{equation}
say. Let $x_{ij}= g_{ij}(p)p^{-s_i}$ ($1\le i\le r$, $1\le j \le t$). Then by \eqref{pol_id},
\begin{equation*} 
\prod_{j=1}^t \Bigl(1-\frac{g_{ij}(p)}{p^{s_i}} \Bigr) = \prod_{j=1}^t (1-x_{ij})= 1-\sum_{j=1}^t (-1)^{j-1} e_j(x_{i1},\ldots,x_{it})= 1-y_i,
\end{equation*}
with $1\le i\le r$, where 
\begin{equation*} 
y_i:= \sum_{j=1}^t (-1)^{j-1} e_j(x_{i1},\ldots,x_{it}) = \sum_{j=1}^t (-1)^{j-1} \sum_{1\le \ell_1<\cdots < \ell_j\le r} x_{i\ell_1}\cdots x_{i\ell_j}  
\end{equation*}
\begin{equation} \label{def_y_i}
= \sum_{j=1}^t \frac{(-1)^{j-1}}{p^{js_i}} \sum_{1\le \ell_1<\cdots < \ell_j\le r} g_{i\ell_1}(p)\cdots g_{i\ell_j}(p)  
= \sum_{j=1}^t \frac{(-1)^{j-1}}{p^{js_i}} G_{ij}(p).  
\end{equation}

Therefore, by applying Lemma \ref{Lemma} for $y_1,\ldots,y_r$ we obtain that  the expression $K(p)$ under the product $\prod_p$ in \eqref{express} is
\begin{equation*}
K(p)= (1-r)\prod_{i=1}^r (1-y_i) + \sum_{i=1}^r \prod_{\substack{k=1\\ k\ne i}}^r (1-y_k) = 1+ \sum_{i=2}^r (-1)^{i-1}(i-1)e_i(y_1,\ldots,y_r)    
\end{equation*}
\begin{equation*}
= 1+ \sum_{i=2}^r (-1)^{i-1}(i-1) \sum_{1\le \ell_1 <\cdots < \ell_i\le r} y_{\ell_1}\cdots y_{\ell_i}
\end{equation*}
\begin{equation*}
= 1+ \sum_{i=2}^r (-1)^{i-1} (i-1) \sum_{1\le \ell_1 <\cdots < \ell_i\le r} \biggl(\sum_{a_{\ell_1}=1}^t \frac{(-1)^{a_{\ell_1}-1}}{p^{a_{\ell_1}s_{\ell_1}}} G_{ia_{\ell_1}}(p)\biggr) \cdots 
\biggl( \sum_{a_{\ell_i}=1}^t \frac{(-1)^{a_{\ell_i}-1}}{p^{a_{\ell_i}s_{\ell_i}}} G_{ia_{\ell_i}}(p)\biggr)
\end{equation*}
\begin{equation*}
= 1- \sum_{i=2}^r (i-1) \sum_{1\le \ell_1<\cdots <\ell_i\le r} \sum_{a_{\ell_1},\ldots,a_{\ell_i}=1}^t 
\frac1{p^{a_{\ell_1}s_{\ell_1}+\cdots +a_{\ell_i}s_{\ell_i}}} \prod_{m=1}^i (-1)^{a_{\ell_m}} G_{i a_{\ell_m}}(p),
\end{equation*}
by \eqref{def_y_i}, ending the proof.
\end{proof}

\begin{remark} {\rm Identity \eqref{id_Th_2} shows that under the assumptions of Theorem \ref{Th_2} we have the convolutional identity
\begin{equation} \label{convo_id_2}
f_1(n_1)\cdots f_r(n_r)\vartheta(n_1,\ldots,n_r) = \sum_{d_1e_1=n_1,\ldots, d_re_r=n_r} f_1(d_1)\cdots f_r(d_r)  
\overline{F}_{f_1,\ldots,f_r}(e_1,\ldots,e_r),
\end{equation}
where $\overline{F}_{f_1,\ldots,f_r}$ is the multiplicative function defined for prime powers $p^{a_1},\ldots,p^{a_r}$ ($a_1,\ldots$, $a_r\ge 0$, not all zero) by
\begin{equation*}
\overline{F}_{f_1,\ldots,f_r}(p^{a_1},\ldots,p^{a_r}) = \begin{cases} (1-i) \prod_{m=1}^i (-1)^{a_{\ell_m}} G_{i a_{\ell_m}}(p), & \text{if 
there exists $2\le i\le r$ and there} \\ & \text{exist $1\le \ell_1,\ldots,\ell_i\le r$ such that} \\ & \text{$1\le a_{\ell_1},\ldots,a_{\ell_i}\le t$;} \\ 0, & 
\text{otherwise.} 
\end{cases}
\end{equation*}
}
\end{remark}

\begin{corollary} \label{Cor_tau_gen_2}
Let $r\ge 2$ and let $t_i\ge 2$ \textup{($1\le i\le r$)} be fixed integers. If $s_i\in \C$, $\Re s_i>1$ \textup{($1\le i\le r$)},
then
\begin{equation*}
\sum_{\substack{n_1,\ldots,n_r=1\\ \gcd(n_i,n_j)=1\, (i\ne j)}}^{\infty} 
\frac{\tau_{t_1}(n_1)\cdots \tau_{t_r}(n_r)}{n_1^{s_1}\cdots n_r^{s_r}} = \zeta^{t_1}(s_1)\cdots \zeta^{t_r}(s_r) 
\overline{\Delta}(\tau_{t_1},\ldots,\tau_{t_r},s_1,\ldots,s_r),
\end{equation*}
where
\begin{align} 
& \overline{\Delta}(\tau_{t_1},\ldots,\tau_{t_r},s_1,\ldots,s_r) 
=  \nonumber \\
& \quad\quad \prod_p \biggl(1- \sum_{i=2}^r (i-1) \sum_{1\le \ell_1<\cdots <\ell_i\le r} \sum_{a_{\ell_1},\ldots,a_{\ell_i}=1}^t 
\frac1{p^{a_{\ell_1}s_{\ell_1}+\cdots +a_{\ell_i}s_{\ell_i}}} \prod_{m=1}^i (-1)^{a_{\ell_m}} \binom{t_i}{a_{\ell_m}}\biggr).
\label{overline_Delta_tau}
\end{align}
\end{corollary}

\begin{proof} Apply Theorem \ref{Th_2} in the case $g_{ij}(p)=1$ for $1\le j\le t_i$ and $g_{ij}(p)=0$ for $t_i+1\le j\le t$ ($1\le i\le r$). 
\end{proof}

\begin{corollary} 
Let $r\ge 2$. If $s_i\in \C$, $\Re s_i>1$ \textup{($1\le i\le r$)}, then 
\begin{align*}
& \sum_{\substack{n_1,\ldots,n_r=1\\ \gcd(n_i,n_j)=1\, (i\ne j)}}^{\infty} 
\frac{\tau(n_1)\cdots \tau(n_r)}{n_1^{s_1}\cdots n_r^{s_r}} = \zeta^2(s_1)\cdots \zeta^2(s_r)  \\
& \quad\quad \times \prod_p \biggl(1- \sum_{i=2}^r (i-1) \sum_{1\le \ell_1<\cdots <\ell_i\le r} \sum_{1\le a_{\ell_1},\ldots,a_{\ell_i}\le 2} 
\frac{(-2)^{\# \{1\le m\le i:\, a_{\ell_m}=1 \}}}{p^{a_{\ell_1}s_{\ell_1}+\cdots +a_{\ell_i}s_{\ell_i}}} \biggr).
\end{align*}
\end{corollary}

\begin{proof} Apply Corollary \ref{Cor_tau_gen_2} for $t_i=2$ ($1\le i\le r$).
\end{proof}

For $r=2$ this gives \eqref{tau_2} and for $r=3$ we have
\begin{align} 
& \sum_{\substack{n_1,n_2,n_3=1\\ \gcd(n_1,n_2)=\gcd(n_1,n_3)=\gcd(n_2,n_3)=1}}^{\infty} 
\frac{\tau(n_1)\tau(n_2)\tau(n_3)}{n_1^{s_1}n_2^{n_2}n_3^{s_3}} = \nonumber \\
& \quad\quad \zeta^2(s_1)\zeta^2(s_2)\zeta^2(s_3) \times \prod_p \Bigl(1-\frac{4}{p^{s_1+s_2}} + \frac{2}{p^{2s_1+s_2}} +\frac{2}{p^{s_1+2s_2}} -\frac1{p^{2s_1+2s_2}} 
-\frac{4}{p^{s_1+s_3}} + \frac{2}{p^{2s_1+s_3}}  \Bigr. \nonumber \\
& \quad\quad\quad \Bigl. +\frac{2}{p^{s_1+2s_3}} -\frac1{p^{2s_1+2s_3}} -\frac{4}{p^{s_2+s_3}} + \frac{2}{p^{2s_2+s_3}} +\frac{2}{p^{s_2+2s_3}} -\frac1{p^{2s_2+2s_3}}  +\frac{16}{p^{s_1+s_2+s_3}} - \frac{8}{p^{2s_1+s_2+s_3}}
  \Bigr.  \nonumber \\
& \quad\quad\quad \Bigl. -\frac{8}{p^{s_1+2s_2+s_3}} -\frac{8}{p^{s_1+s_2+2s_3}} 
 +\frac{4}{p^{2s_1+2s_2+s_3}}  \Bigr. 
\Bigl. + \frac{4}{p^{2s_1+s_2+2s_3}} +\frac{4}{p^{s_1+2s_2+2s_3}} -\frac{2}{p^{2s_1+2s_2+2s_3}} 
\Bigr).
\label{prod_tau_pairw_prime_3}
\end{align}

If we compare the infinite product \eqref{prod_tau_pairw_prime_3} to \eqref{prod_tau_prime_3}, then we can see that in
\eqref{prod_tau_prime_3} we only have exponents of $p$ of form $a_1s_1+a_2s_2+a_3s_3$ with $1\le a_1,a_2,a_3\le 2$, while 
in \eqref{prod_tau_pairw_prime_3} the exponents of $p$ are $a_1s_1+a_2s_2+a_3s_3$ with $0\le a_1,a_2,a_3\le 2$ and with at least
two nonzero values $a_1,a_2,a_3$. Similar in the general case, according to Theorems \ref{Th_1} and \ref{Th_2}.

It is possible to derive a common generalization of Theorems \ref{Th_1} and \ref{Th_2} by considering $k$-wise relatively
prime integers. Let $r\ge k\ge 2$ be fixed integers. The positive integers
$n_1,\ldots,n_r$ are called $k$-wise relatively prime if any $k$ of
them are relatively prime, that is, $\gcd(n_{i_1},\ldots,n_{i_k})=1$
for every $1\le i_1<\cdots<i_k\le r$. In particular, in the case
$k=2$ the integers are pairwise relatively prime and for $k=r$ they
are mutually relatively prime.  Let $\varrho_k$
denote the characteristic function of the set of $r$-tuples of
positive integers with $k$-wise relatively prime components.
Hence, $\varrho_r=\varrho$ and $\varrho_2=\vartheta$, with
our previous notation.

Here we confine ourselves to the case $t=1$, that is, the functions $f_1,\ldots,f_r$ are completely
multiplicative.

\begin{theorem} \label{Th_Dir_ser} Let $r\ge k\ge 2$ and let $f_1,\ldots,f_r:\N \to \C$ be completely
multiplicative functions. Then
\begin{align*}
& \sum_{n_1,\ldots,n_r=1}^{\infty}
\frac{f_1(n_1)\cdots f_r(n_r) \varrho_k(n_1,\ldots,n_r)}{n_1^{s_1}\cdots n_r^{s_r}}=
D(f_1,s_1)\cdots D(f_r,s_r)  \\
& \quad\quad \times \prod_p \biggl( 1- \sum_{i=k}^r (-1)^{i-k}
\binom{i-1}{k-1} \sum_{1\le \ell_1<\cdots < \ell_i\le r} \frac{f_{\ell_1}(p)\cdots f_{\ell_i}(p)}{p^{s_{\ell_1}+\cdots +s_{\ell_i}}} \biggr),
\end{align*}
\end{theorem} 

\begin{proof} For fixed $k$ the function $\varrho_k$ is multiplicative. Also, for prime powers $p^{a_1},\ldots$, $p^{a_r}$ 
($a_1\,\ldots,a_r\ge 0$) we have $\varrho_k(p^{a_1},\ldots, p^{a_r})=1$ if and only there are at most $k-1$ values $a_i\ge 1$.  
Now the proof is similar to the proofs of Theorems \ref{Th_1} and \ref{Th_2}. In the case $f_1(n)=\cdots =f_r(n)=1$ ($n\in \N$) this 
result and its detailed proof are given in \cite[Th.\ 2.1]{Tot2016}.
\end{proof}

\section{Related asymptotic formulas} \label{Section_Asymptotic_formulas}

The above identities can be used to obtain asymptotic formulas with remainder terms for certain related sums.
As examples, we point out the following formulas.

\begin{theorem} \label{Th_asympt_1} Let $r\ge 2$ and let $t_i\ge 2$ \textup{($1\le i\le r$)} be fixed integers.
Then for every $\varepsilon>0$,
\begin{equation*}
\sum_{\substack{n_1,\ldots,n_r\le x \\ \gcd(n_1,\ldots,n_r)=1}} \tau_{t_1}(n_1) \cdots \tau_{t_r}(n_r) = x^r Q(\log x)+ O(x^{r-1+\max_{1\le i\le r} \vartheta_{t_i} +\varepsilon}),
\end{equation*}
where $Q(u)$ is a polynomial in $u$ of degree $t_1+\cdots+t_r-r$ having the leading coefficient 
\begin{equation*}
\frac{\Delta(\tau_{t_1},\ldots,\tau_{t_r},1,\ldots,1)}{(t_1-1)!\cdots (t_r-1)!},
\end{equation*}
where $\Delta(\tau_{t_1},\ldots,\tau_{t_r},1,\ldots,1)$ is obtained from \eqref{Delta_tau}
for $s_1=\cdots=s_r=1$, and $\vartheta_{t_i}$ are the exponents in the Piltz divisor problems for $\tau_{t_i}$, namely 
\begin{equation} \label{Piltz}
\sum_{n\le x}  \tau_{t_i}(n) = x P_{t_i}(\log x)+ O(x^{\vartheta_{t_i} +\varepsilon}),
\end{equation}
with some polynomials $P_{t_i}(u)$ in $u$ of degree $t_i-1$ having the leading coefficients $1/(t_i-1)!$ \textup{($1\le i \le r$)}. 
\end{theorem}

\begin{proof}  We have, according to the convolutional identity \eqref{convo_id_1},
\begin{equation*} 
\sum_{\substack{n_1,\ldots,n_r\le x \\ \gcd(n_1,\ldots,n_r)=1}} \tau_{t_1}(n_1) \cdots \tau_{t_r}(n_r)=
\sum_{d_1e_1=n_1\le x,\ldots,d_re_r=n_r\le x} \tau_{t_1}(d_1) \cdots \tau_{t_r}(d_r)  F_{\tau_{t_1},\ldots,\tau_{t_r}}(e_1,\ldots,e_r),
\end{equation*}
\begin{equation*} 
= \sum_{e_1,\ldots, e_r \le x} F_{\tau_{t_1},\ldots,\tau_{t_r}}(e_1,\ldots,e_r) \sum_{d_1\le x/e_1} \tau_{t_1}(d_1) \cdots \sum_{d_r\le x/e_r} \tau_{t_r}(d_r).
\end{equation*}

Now by using formulas \eqref{Piltz} and the fact that the infinite product $\Delta(\tau_{t_1},\ldots,\tau_{t_r},s_1,\ldots,s_r)$ given by \eqref{Delta_tau} is absolutely convergent provided that
$s_i\in \C$, $\Re s_i>0$ ($1\le i \le r$), $\Re (s_1+\cdots +s_r)>1$, we obtain the desired formula. 
For the details see the proof of \cite[Th.\ 3.3]{TotZha2018}, which is a generalization of the present result. 
\end{proof}

\begin{corollary}  Let $r\ge 2$. Then for every $\varepsilon>0$,
\begin{equation*}
\sum_{\substack{n_1,\ldots,n_r\le x \\ \gcd(n_1,\ldots,n_r)=1}} \tau(n_1)\cdots \tau(n_r) = x^r T(\log x)+ O(x^{r-1+\theta +\varepsilon})
\end{equation*}
where $T(u)$ is a polynomial in $u$ of degree $r$ having the leading coefficient $K_r$, where
\begin{equation*}
K_r= \prod_p \Bigl(1- \Bigl(\frac{2p-1}{p^2}\Bigr)^r \Bigr)=
\prod_p \biggl(1 - \sum_{i=0}^r (-1)^i \binom{r}{i} \frac{2^{r-i}}{p^{r+i}} \biggr);
\end{equation*}
in particular, 
\begin{align} 
K_2 &= \prod_p \Bigl(1 - \frac{4}{p^2}+\frac{4}{p^3}-\frac1{p^4} \Bigr), \label{K_2} \\
K_3 &= \prod_p \Bigl(1 - \frac{8}{p^3}+\frac{12}{p^4}-\frac{6}{p^5} +\frac1{p^6}\Bigr), \nonumber
\end{align}
and $\theta$ is the exponent in Dirichlet's divisor problem. 
\end{corollary}

\begin{proof} Apply Theorem \ref{Th_asympt_1} in the case $t_i=2$ ($\le i\le r$). The representation of $K_r$ follows 
from \eqref{D_form_one} for $g_{ij}(p)=1$ ($1\le i\le r$, $1\le j\le 2$).
\end{proof}

We note that in the case $r=2$ this result has been proved in \cite[Lemma\, 3.3]{NowTot2014} by analytic methods,
with a weaker error term.

\begin{theorem} \label{Th_asympt_2} Let $r\ge 2$ and let $t_i\ge 2$ \textup{($1\le i\le r$)} be fixed integers.
Then for every $\varepsilon>0$,
\begin{equation*}
\sum_{\substack{n_1,\ldots,n_r\le x \\ \gcd(n_i,n_j)=1 \, (i\ne j)}} \tau_{t_1}(n_1) \cdots \tau_{t_r}(n_r) 
= x^r \overline{Q}(\log x)+ 
O(x^{r-1+\max_{1\le i\le r} \vartheta_{t_i} +\varepsilon}),
\end{equation*}
where $\overline{Q}(u)$ is a polynomial in $u$ of degree $t_1+\cdots+t_r-r$ having the leading coefficient
\begin{equation*}
\frac{\overline{\Delta}(\tau_{t_1},\ldots,\tau_{t_r},1,\ldots,1)}{(t_1-1)!\cdots (t_r-1)!},
\end{equation*}
where $\overline{\Delta}(\tau_{t_1},\ldots,\tau_{t_r},1,\ldots,1)$ is obtained from \eqref{overline_Delta_tau}
for $s_1=\cdots=s_r=1$, and $\vartheta_{t_i}$ are the exponents in the Piltz divisor problems for $\tau_{t_i}$ \textup{($1\le i \le r$)}. 
\end{theorem}

\begin{proof} Similar to the proof of Theorem \ref{Th_asympt_1}. By the convolutional identity \eqref{convo_id_2} we have
\begin{equation*} 
\sum_{\substack{n_1,\ldots,n_r\le x \\ \gcd(n_i,n_j)=1\, (i\ne j)}} \tau_{t_1}(n_1) \cdots \tau_{t_r}(n_r)=
\sum_{d_1e_1=n_1\le x,\ldots,d_re_r=n_r\le x} \tau_{t_1}(d_1) \cdots \tau_{t_r}(d_r)  \overline{F}_{\tau_{t_1},\ldots,\tau_{t_r}}(e_1,\ldots,e_r),
\end{equation*}
\begin{equation*} 
= \sum_{e_1,\ldots, e_r \le x} \overline{F}_{\tau_{t_1},\ldots,\tau_{t_r}}(e_1,\ldots,e_r) \sum_{d_1\le x/e_1} \tau_{t_1}(d_1) \cdots \sum_{d_r\le x/e_r} \tau_{t_r}(d_r). 
\end{equation*}

Now use formulas \eqref{Piltz} and the fact that the infinite product $\overline{\Delta}(\tau_{t_1},\ldots,\tau_{t_r},s_1,\ldots,s_r)$ given by 
\eqref{overline_Delta_tau} is absolutely convergent provided that
$s_i\in \C$, $\Re s_i>0$ ($1\le i \le r$), $\Re (s_i+s_j)>1$ ($1\le i<j\le r$). 
This is also a special case of \cite[Th.\ 3.3]{TotZha2018}.
\end{proof}

\begin{corollary}  Let $r\ge 2$. Then for every $\varepsilon>0$,
\begin{equation*}
\sum_{\substack{n_1,\ldots,n_r\le x \\ \gcd(n_i,n_j)=1\, (i\ne j)}} \tau(n_1)\cdots \tau(n_r) = x^r \overline{T}(\log x)+ O(x^{r-1+\theta +\varepsilon})
\end{equation*}
where $\overline{T}(u)$ is a polynomial in $u$ of degree $r$ having the leading coefficient $\overline{K}_r$, where
\begin{equation*}
\overline{K}_r= \prod_p \Bigl(1-\frac1{p}\Bigr)^{2(r-1)} \Bigl(1+\frac{(r-1)(2p-1)}{p^2}\Bigr),
\end{equation*}
in particular, $\overline{K}_2=K_2$ given by \eqref{K_2},
\begin{equation*}
\overline{K}_3= \prod_p \Bigl(1 - \frac{12}{p^2}+\frac{28}{p^3}-\frac{27}{p^4} +\frac{12}{p^5}-\frac{2}{p^6}\Bigr),
\end{equation*}
and $\theta$ is the exponent in Dirichlet's divisor problem. 
\end{corollary}

\begin{proof} Apply Theorem \ref{Th_asympt_2} in the case $t_i=2$ ($\le i\le r$). The representation of $\overline{K}_r$ 
follows from \eqref{express} for $g_{ij}(p)=1$ ($1\le i\le r$, $1\le j\le 2$).
\end{proof}

\end{document}